\documentclass[11pt]{article}

\usepackage[margin=1in]{geometry}
\usepackage[T1]{fontenc}
\usepackage[utf8]{inputenc}
\usepackage{lmodern}
\usepackage{amsmath}
\usepackage{amssymb}
\usepackage{amsthm}
\usepackage{mathtools}
\usepackage{graphicx}
\usepackage{booktabs}
\usepackage{multirow}
\usepackage{array}
\usepackage{makecell}
\usepackage{float}
\usepackage{caption}
\usepackage{subcaption}
\usepackage{algorithm}
\usepackage{algpseudocode}
\usepackage{xcolor}
\usepackage{url}
\usepackage[colorlinks=true,
            linkcolor=blue,
            citecolor=blue,
            urlcolor=blue]{hyperref}
\usepackage[nameinlink,capitalize]{cleveref}
\usepackage{listings}

\definecolor{codegray}{gray}{0.95}
\definecolor{codegreen}{rgb}{0.0,0.45,0.0}
\definecolor{codepurple}{rgb}{0.45,0.0,0.55}

\definecolor{revcolor}{RGB}{180,0,0} 

\lstdefinestyle{pythonstyle}{
    language=Python,
    backgroundcolor=\color{codegray},
    basicstyle=\ttfamily\small,
    keywordstyle=\color{blue},
    commentstyle=\color{codegreen},
    stringstyle=\color{codepurple},
    numbers=left,
    numberstyle=\tiny,
    stepnumber=1,
    numbersep=8pt,
    showstringspaces=false,
    breaklines=true,
    frame=single,
    tabsize=4,
    captionpos=b
}

\usepackage{tikz}
\usepackage{pgfplots}
\pgfplotsset{compat=1.18}
\newtheorem{theorem}{Theorem}
\newtheorem{proposition}{Proposition}
\newtheorem{lemma}{Lemma}
\newtheorem{corollary}{Corollary}

\theoremstyle{definition}

\newtheorem{problem}{Problem}

\title{Strong NP-Hardness of Rank-2 Nonnegative Matrix Factorization}

\author{Arnaud Vandaele\thanks{Department of Mathematics and Operational Research, University of Mons, Mons, Belgium. \\ Email: arnaud.vandaele@umons.ac.be.}}
\date{}

\begin{document}
\maketitle

\begin{abstract}
Nonnegative matrix factorization (NMF) approximates a nonnegative matrix by the product of two nonnegative matrices with a prescribed inner dimension~$r$.
The existence of an exact NMF can be decided in polynomial time for every fixed~$r$, and a best rank-$1$ NMF in the Frobenius norm can be obtained from a truncated singular value decomposition.
We prove that computing a best rank-$2$ NMF in the Frobenius norm is strongly NP-hard.
Our proof uses a polynomial-time reduction from \textsc{Clique} that produces symmetric positive-definite, completely positive
matrices with polynomially bounded integer entries.
The same reduction also establishes strong NP-hardness of rank-$2$ symmetric NMF.
\end{abstract}

\section{Introduction}\label{sec:intro}
 
Given a nonnegative matrix $X\in\mathbb{R}^{m\times n}_{\geq0}$ and an integer $r\geq1$, NMF seeks nonnegative factors $W\in\mathbb{R}^{m\times r}_{\geq0}$ and $H\in\mathbb{R}^{r\times n}_{\geq0}$ whose product approximates~$X$.
Introduced by Paatero and Tapper~\cite{paatero1994positive} and popularized by Lee and Seung~\cite{lee1999learning}, NMF has become a standard tool for extracting interpretable, parts-based representations of nonnegative data; see the monograph~\cite{gillis2020nonnegative} and the references therein.
Measuring the error in the Frobenius norm leads to the optimization problem
\begin{equation}\label{eq:nmf}
\min_{W\in\mathbb{R}^{m\times r},\,H\in\mathbb{R}^{r\times n}}\ \|X-WH\|_F^2\quad\text{s.t.}\quad W\geq0,\ H\geq0.
\end{equation}
Without the nonnegativity constraints, \eqref{eq:nmf} is solved by a truncated singular value decomposition (SVD).
Indeed, by the Eckart--Young--Mirsky theorem, the truncated SVD provides a best rank-$r$ approximation in every unitarily invariant norm~\cite{eckart1936approximation,mirsky1960symmetric}, and it can be computed by iterative methods at a cost of $O(mnr)$ operations per iteration~\cite{golub1965calculating}.
With the nonnegativity constraints, the truncated SVD is in general infeasible, since its factors may have negative entries, and \eqref{eq:nmf} becomes a nonconvex problem.
Standard algorithms, such as multiplicative updates and block coordinate descent methods based on alternating nonnegative least squares, are local methods: at best, they guarantee convergence to stationary points~\cite{gillis2020nonnegative}.
This raises a natural algorithmic question:
\begin{center}
\textit{for which values of~$r$ can a global minimizer of~\eqref{eq:nmf} be computed in polynomial time?}
\end{center}

The case $r=2$ is of particular interest.
Rank-$2$ NMF is the building block of hierarchical clustering methods that recursively split the data into two clusters, for instance, in document clustering~\cite{kuang2013fast}, hyperspectral imaging~\cite{gillis2014hierarchical}, and the characterization of brain-tumor tissues~\cite{li2013hierarchical,sauwen2015hierarchical}.
For $r=2$, the subproblems of alternating schemes are particularly simple: when one factor is fixed, the problem decouples into nonnegative least squares problems in two variables, which can be solved exactly~\cite{kuang2013fast}.
More recently, Lindy, Noferini, and Van Dooren~\cite{lindy2025rank} studied the structure of rank-$2$ nonnegative factorizations and proposed an SVD-based initialization for alternating nonnegative least squares.
They also noted that whether computing a best rank-two NMF is NP-hard was, to their knowledge, an open problem.
In this paper, we prove that this problem is strongly NP-hard.
 
\subsection{Known complexity results}\label{subsec:nutshell}
 
The complexity of NMF depends on three features: whether the factorization is exact or approximate, whether the inner dimension~$r$ is fixed or part of the input, and whether the input matrix is nonnegative.
These distinctions lead to the three problems defined formally below: exact NMF (\textsc{ExactNMF}), approximate NMF in the Frobenius norm (\textsc{FroNMF}), and its extension to input matrices of arbitrary signs, called nonnegative factorization (\textsc{NF})~\cite{gillis2008nonnegative}.
 
\begin{problem}[\textsc{ExactNMF}]
Given a matrix $X\in\mathbb{R}_{\geq 0}^{m\times n}$ and an integer $r\leq \min(m,n)$, determine whether there exist matrices $W\in\mathbb{R}_{\geq 0}^{m\times r}$ and $H\in\mathbb{R}_{\geq 0}^{r\times n}$ such that $X=WH$.
\end{problem}

\begin{problem}[\textsc{FroNMF}]
Given a matrix $X\in\mathbb{R}_{\geq 0}^{m\times n}$ and an integer $r\leq \min(m,n)$, solve
\[
\min_{W\in\mathbb{R}^{m\times r}, H\in\mathbb{R}^{r\times n}} \|X-WH\|_F^2 \quad \text{s.t.}\quad  W\geq 0, ~H\geq 0.
\]
\end{problem}

\begin{problem}[\textsc{NF}]
Given a matrix $M\in\mathbb{R}^{m\times n}$ and an integer $r\leq \min(m,n)$, solve
\[
\min_{W\in\mathbb{R}^{m\times r}, H\in\mathbb{R}^{r\times n}} \|M-WH\|_F^2 \quad \text{s.t.}\quad  W\geq 0, ~H\geq 0.
\]
\end{problem}
 
For a fixed integer~$r$, we denote by \textsc{Rank-$r$-NMF} and \textsc{Rank-$r$-NF} the restrictions of \textsc{FroNMF} and \textsc{NF} to instances with inner dimension~$r$.
 
\paragraph{Exact NMF.}
Vavasis~\cite{vavasis2010complexity} proved that \textsc{ExactNMF} is NP-hard when $r$ is part of the input.
In contrast, Arora et al.~\cite{arora2012computing} showed that it can be solved in polynomial time for a fixed $r$ by providing an algorithm with a running time of $(mn)^{O(r^2 2^r)}$, which was later improved to $(mn)^{O(r^2)}$ by Moitra~\cite{moitra2016almost}.
 
\paragraph{Approximate NMF.}
When $r$ is part of the input, \textsc{FroNMF} is NP-hard as an immediate consequence of the hardness of \textsc{ExactNMF}: the optimal approximation error is zero if and only if an exact nonnegative factorization of inner dimension $r$ exists.
This argument, however, gives no information when $r$ is fixed.

When $r$ is fixed, Arora et al.~\cite{arora2012computing} also obtained an approximation guarantee, assuming that a sufficiently accurate nonnegative factorization exists.
Specifically, if $X$ admits nonnegative factors of inner dimension $r$ with $\|X-WH\|_F\leq\varepsilon\|X\|_F$, their algorithm computes nonnegative factors $\widetilde W,\widetilde H$ of the same inner dimension such that
\[
\|X-\widetilde W\widetilde H\|_F \leq O\left(r^{1/4}\sqrt{\varepsilon}\right)\|X\|_F
\]
in time $2^{\operatorname{poly}(r\log(1/\varepsilon))}\operatorname{poly}(m,n)$.
Thus, the running time is polynomial in the matrix dimensions for fixed $r$ and $\varepsilon$.
This guarantee provides an upper bound on the reconstruction error but does not ensure that the returned factors attain the smallest possible error for a given inner dimension.

Computing a global optimum is nevertheless possible in polynomial time for $r=1$.
Indeed, by the Perron-Frobenius theorem, a nonnegative matrix admits a pair of nonnegative leading singular vectors, so the truncated SVD already provides an optimal nonnegative rank-one approximation, and \textsc{Rank-1-NMF} can be solved in polynomial time.
For $r=2$, the same argument applies whenever the best rank-$2$ approximation $X_2$ given by the truncated SVD is nonnegative.
Indeed, every nonnegative matrix of rank at most two admits a nonnegative factorization of inner dimension two~\cite{thomas1974rank}, which can be computed efficiently, so that $X_2$ leads to an optimal NMF. 

\paragraph{Nonnegative factorization.}
The situation changes if the matrix to be approximated is allowed to have negative entries.
Indeed, Gillis and Glineur~\cite{gillis2008nonnegative,gillis2014continuous} showed, by a reduction from the maximum-edge biclique problem, that \textsc{Rank-1-NF} is NP-hard.
Thus, while \textsc{Rank-1-NMF} is polynomial-time solvable,
allowing arbitrary signs in the input matrix is sufficient to make the corresponding rank-$1$ approximation problem NP-hard.
Belachew and Gillis~\cite{belachew2017solving} showed that the symmetric variant of this problem, which seeks a best approximation of a symmetric matrix by a matrix $\lambda xx^\top$ with $\lambda\geq0$ and $x\geq0$, encodes the clique number of a graph, through a continuous characterization in the spirit of the Motzkin--Straus theorem~\cite{motzkin1965maxima}.
 
\subsection{Contributions}\label{subsec:contribution}
The results above leave rank-$2$ NMF as the first nontrivial fixed-rank case.
Moreover, if the best unconstrained rank-$2$ approximation $X_2$ is entrywise nonnegative, then $X_2$ is also optimal for \textsc{Rank-2-NMF}.
The remaining cases are those in which $X_2$ contains negative entries.
Our main result is the following.

\begin{theorem}\label{thm:r2hardness}
\textsc{Rank-2-NMF} is strongly NP-hard.
\end{theorem}

Our proof is based on a polynomial-time reduction from \textsc{Clique}.
The reduction encodes the clique number in the optimal nonnegative rank-one approximation error of a residual matrix.
A quantitative perturbation argument then shows that allowing both rank-one terms to vary cannot overcome the error gap between consecutive clique sizes.

Theorem~\ref{thm:r2hardness} contrasts with the exact case: \textsc{ExactNMF} is polynomial-time solvable for every fixed inner dimension, whereas computing a globally optimal NMF is already strongly NP-hard for inner dimension two.
Table~\ref{tab:nmf-complexity} summarizes the complexity results.
 
\begin{table}[ht]
\centering
\renewcommand{\arraystretch}{1.3}
\begin{tabular}{l|c|c|c}
\hline
 & \textsc{ExactNMF} & \textsc{FroNMF} & \textsc{NF}\\
\hline
$r$ fixed
 & polynomial~\cite{arora2012computing,moitra2016almost}
 & \makecell{polynomial for $r=1$\\ strongly NP-hard for $r=2$\\ (Theorem~\ref{thm:r2hardness})}
 & \makecell{NP-hard for $r=1$\\ \cite{gillis2008nonnegative,gillis2014continuous}}\\
$r$ part of the input
 & NP-hard~\cite{vavasis2010complexity} & NP-hard & NP-hard\\
\hline
\end{tabular}
\caption{Complexity landscape of \textsc{ExactNMF}, \textsc{FroNMF}, and \textsc{NF}.}
\label{tab:nmf-complexity}
\end{table}
 
Consequently, unless $\mathrm{P}=\mathrm{NP}$, no polynomial-time algorithm can guarantee a globally optimal rank-two NMF for every input, even though the subproblems of alternating nonnegative least squares can be solved exactly.
The hardness persists under strong restrictions on the input matrices.
Indeed, the matrices produced by the reduction are symmetric and have nonnegative integer entries bounded by a polynomial in the number of vertices of the graph.
Moreover, they are completely positive (Proposition~\ref{prop:complete-positivity}): they admit an exact symmetric factorization $X=WW^\top$ with $W\geq0$, but only with an inner dimension larger than two.
In addition, the factors constructed for yes-instances satisfy $H=W^\top$.
As a consequence, the same reduction applies to symmetric NMF (SymNMF), which seeks an approximation $X\approx WW^\top$ with $W\geq0$.
Computing a best rank-two SymNMF is therefore strongly NP-hard, even for completely positive input matrices (Corollary~\ref{cor:rank2-symnmf} in Section~\ref{sec:symnmf}).
 
\paragraph{Concurrent AI-generated work.}
A proof of NP-hardness for rank-$2$ NMF was posted by Colbrook in the \textit{OpenProblemsInNLA} repository on 11 September 2026\footnote{\url{https://github.com/ajt60gaibb/OpenProblemsInNLA/blob/main/references/colbrook-factorization-2026-09-11/manuscripts/NM-03_rank_two_approximation_hardness.pdf}}.
However, this paper was AI-generated and the repository reports AI-agent verification rather than external human peer review.
Our work on the gadget and several components of the proof began before Colbrook's manuscript was posted online.
The present construction is a continuation of earlier results connecting rank-one matrix approximation problems with combinatorial problems on graphs~\cite{belachew2017solving,gillis2011low,gillis2014continuous,gillis2018complexity}.
 
\subsection{Organization of the paper}

Section~\ref{sec:rank-two-hardness} presents the reduction from \textsc{Clique}.
Subsections~\ref{subsec:r2nmf-construction} and~\ref{subsec:idea} present the construction and explain how the clique number is encoded in the optimal nonnegative rank-one approximation of a residual matrix. Subsection~\ref{subsec:r2nmf-general-bounds} establishes perturbation bounds that control the improvement possible when both rank-one terms are allowed to vary.
These bounds are then used in Subsection~\ref{subsec:r2nmf-correctness} to prove the correctness of the reduction.
Subsection~\ref{subsec:numerical-illustration} illustrates the reduction with a numerical example.
Finally, Section~\ref{sec:symnmf} establishes the complete positivity of the constructed matrices and extends the hardness result to rank-$2$ symmetric NMF.

\subsection{Notation}
\label{subsec:notation}
For vectors, $\|\cdot\|_2$ denotes the Euclidean norm; for matrices, $\|\cdot\|_F$ denotes the Frobenius norm, associated with the inner product $\langle U,V\rangle=\operatorname{tr}(U^\top V)$.
Inequalities between vectors or matrices are understood entrywise; accordingly, $\mathbb{R}^{m\times n}_{\geq0}$ denotes the set of $m\times n$ real matrices with nonnegative entries, and similarly for $\mathbb{Q}^{m\times n}_{\geq0}$ and $\mathbb{Z}^{m\times n}_{\geq0}$.
In contrast, $M\succeq0$ and $M\succ0$ mean that $M$ is symmetric positive semidefinite and positive definite, respectively.
The positive and negative parts of a matrix $M$ are $M^+=\max(M,0)$ and $M^-=\max(-M,0)$, with the maximum taken entrywise, so that $M=M^+-M^-$; the same notation is used for vectors.
We denote by $\mathbf{1}_k$ the all-ones vector of $\mathbb{R}^k$, by $\mathbf{1}_{k\times\ell}$ the all-ones $k\times\ell$ matrix, and by $I_k$ the identity matrix of order~$k$.
For a graph $G$, $\omega(G)$ denotes its clique number, that is, the largest number of pairwise adjacent vertices.

\section{NP-hardness of \textsc{Rank-2-NMF}}
\label{sec:rank-two-hardness}
In this section, we study the complexity of the \textsc{Rank-2-NMF} problem by providing a reduction from the \textsc{Clique} problem.
We begin by formally stating the decision versions of both problems.

\noindent
\fbox{\begin{minipage}{\dimexpr\linewidth-2\fboxsep-2\fboxrule\relax}
\textbf{\textsc{Clique}}.

\noindent
\emph{Given:} a graph $G=(V,E)$ and a positive integer $J\leq |V|$.

\noindent
\emph{Question:} does $G$ contain a subset $\mathcal{C}\subseteq V$ such that $|\mathcal{C}|\geq J$ and every two vertices in $\mathcal{C}$ are joined by an edge in $E$?
\end{minipage}}\\

\smallskip

\noindent
\fbox{\begin{minipage}{\dimexpr\linewidth-2\fboxsep-2\fboxrule\relax}
\textbf{\textsc{Rank-2-NMF}}.

\noindent
\emph{Given:} a matrix $X\in\mathbb Q_{\geq 0}^{m\times n}$ and a rational number $T>0$.

\noindent
\emph{Question:} do there exist $W\in\mathbb R_{\geq 0}^{m\times2}$ and $H\in\mathbb R_{\geq 0}^{2\times n}$ such that $\|X-WH\|_F^2\leq T$?
\end{minipage}}\\

\smallskip

\subsection{Construction of the reduction}
\label{subsec:r2nmf-construction}

\textsc{Clique} is NP-complete~\cite{karp1972reducibility}.
Let $(G,J)$ be an instance of \textsc{Clique}, where $G=(V,E)$ is a simple undirected graph with $N=|V|$ vertices and adjacency matrix $A\in\{0,1\}^{N\times N}$.
From $(G,J)$, we construct an instance $(X,T)$ of \textsc{Rank-2-NMF} in polynomial time as follows.
We set 
\begin{equation}\label{eq:reduction}
\begin{gathered}
m = n = N+1, \qquad d=20N^3, \qquad \alpha=1+\frac{1}{N}, \qquad C=dI_{N}+A-\frac{1}{\alpha}\textbf{1}_{N\times N}, \\[1ex]
X=\begin{pmatrix}C+\frac{1}{\alpha}\textbf{1}_{N\times N} & d\textbf{1}_{N}\\d\textbf{1}^T_{N} & \alpha d^2 \end{pmatrix}\in\mathbb{Z}_{\geq0}^{(N+1)\times(N+1)},
\qquad \text{and}\quad T=\|C\|_F^2-\left(d-1+\frac{J}{N+1}\right)^2.
\end{gathered}
\end{equation}
The leading block $C+\frac1\alpha\mathbf{1}_{N\times N}=dI_N+A$ of $X$ is nonnegative, and all entries of $X$ and $T$ are rational numbers of polynomial size.

\subsection{Structure of the instance}
\label{subsec:idea}
Before proving correctness, we explain how the components of the construction in~\eqref{eq:reduction} encode \textsc{Clique}.
In the leading $N\times N$ block of $X$, given by $C + \frac{1}{\alpha}\mathbf{1}_{N\times N}=dI_N+A$, the diagonal entries of the adjacency matrix are set to a large value, $d = 20N^3$, while the adjacency pattern of $A$ is preserved: off-diagonal entries are $1$ for adjacent vertices and $0$ for nonadjacent vertices.
The last row and column of $X$ are chosen so that any rank-$1$ matrix matching them exactly has every entry in its leading $N\times N$ block equal to $1/\alpha$.
Indeed, let $Q\in\mathbb{R}^{(N+1)\times(N+1)}$ be a rank-$1$ matrix such that $Q(:,N+1)=Q(N+1,:)^\top=(d\textbf{1}^\top_N~~\alpha d^2)^\top$, then
\[
Q_{ij} = \frac{Q_{i,N+1}Q_{N+1,j}}{Q_{N+1,N+1}} = \frac{d^2}{\alpha d^2} = \frac{1}{\alpha}, \qquad 1\leq i,j\leq N.
\]
This is the reason for the value $\alpha d^2$ for the lower-right entry of $X$: the values of the last row and column of this rank-one term can be large, depending on $d$, while the contribution to the leading block is independent of $d$.
Since $d$ is large, any sufficiently accurate rank-$2$ approximation of $X$ must approximate the last row and column closely.
These entries therefore naturally suggest a dominant rank-$1$ component
capturing this part of $X$.
Let us denote by $Q=vv^\top$ the rank-one matrix matching the last row and column of $X$ exactly, we have
\[
v=\begin{pmatrix}\frac{1}{\sqrt{\alpha}}\mathbf{1}_N\\\sqrt{\alpha}d\end{pmatrix}, \quad \text{ and }\quad vv^T=\begin{pmatrix}\frac{1}{\alpha}\mathbf{1}_{N\times N} & d\textbf{1}_N\\d\textbf{1}_N^T & \alpha d^2\end{pmatrix}.
\]
Subtracting this matrix $vv^T$ from $X$ gives a residual only in the upper-left leading block:
\begin{equation}\label{eq:decomposition2rank}
X=vv^\top+\begin{pmatrix}C&0\\0&0\end{pmatrix}\quad \text{ where } \quad C=dI_N+A-\frac{1}{\alpha}\mathbf{1}_{N\times N}.
\end{equation}
Since $1/\alpha=N/(N+1)$, we have
\begin{equation}\label{eq:C_details}
    C_{ij}=
    \begin{cases}
        \displaystyle d-\frac{N}{N+1}>0,
            & \text{if } i=j,\\[2mm]
        \displaystyle \frac{1}{N+1}>0,
            & \text{if } i\neq j \text{ and } (i,j)\in E,\\[2mm]
        \displaystyle -\frac{N}{N+1}<0,
            & \text{if } i\neq j \text{ and } (i,j)\notin E.
    \end{cases}
\end{equation}
From~\eqref{eq:C_details}, we see that the choices of $\alpha$ and $d$ in~\eqref{eq:reduction} give $C$ a sign pattern that directly reflects the structure of $G$.
Let us define $C^+=\max(C,0)$ and $C^-=\max(-C,0)$, with the maximum taken entrywise, so that $C=C^+-C^-$:
\begin{equation}\label{eq:C-positive-negative}
C^+ = \left(d-\frac1\alpha\right)I_N+\left(1-\frac{1}{\alpha}\right)A,\qquad C^-= \frac{1}{\alpha}\left(\mathbf{1}_{N\times N}-I_N-A\right).
\end{equation}
We can observe that $C^+$ contains large diagonal residuals and small positive residuals on edges, while $C^-$ has positive entries exactly on nonedges.
The decomposition~\eqref{eq:decomposition2rank} suggests first considering rank-$2$ approximations in which one rank-$1$ term is fixed to $vv^\top$.
The squared approximation error is 
\[
\|X-vv^\top\|_F^2=\|C\|_F^2.
\]
By using another rank-$1$ correction $Z\geq 0$ to the leading block, the objective becomes
\[
\left\|X-vv^\top-\begin{pmatrix}Z&0\\0&0\end{pmatrix}\right\|_F^2 = \|C-Z\|_F^2.
\]
For each entry $(i,j)$, the change in the objective function is 
\[
(C_{ij}-Z_{ij})^2-C_{ij}^2 = Z_{ij}^2-2C_{ij}Z_{ij}.
\]
This means the signs of $C$ indicate where a nonnegative rank-one correction can reduce the error.
Indeed, if $C_{ij}<0$, then $(i,j)$ corresponds to a nonedge, so $X_{ij}=0$ while $(vv^\top)_{ij}=1/\alpha$.
Hence, $vv^\top$ exceeds the zero entry by $1/\alpha$, and any positive contribution $Z_{ij}$ further increases the error.
If $C_{ij}>0$, then $(vv^\top)_{ij}$ is below $X_{ij}$, and adding $Z_{ij}>0$ reduces the error whenever $0<Z_{ij}<2C_{ij}$.

However, these contributions cannot be optimized independently since $Z$ must be rank-$1$.
W.l.o.g., we can write $Z=\lambda xy^\top$, where $\lambda\geq0$, $x,y\geq0$, and $\|x\|_2=\|y\|_2=1$.
The large diagonal of $X$ allows for a simplification: an optimal correction can be chosen to be symmetric.
Indeed, $C$ is symmetric and
\[
C_{ii}-\sum_{j\neq i}|C_{ij}|\geq d-\frac{N}{\alpha}>0,
\]
so $C$ is positive definite.
Thus, a best nonnegative rank-one approximation of $C$ can be chosen symmetric and written as $\lambda xx^\top$; see Lemma~\ref{lem:rank1-symmetrization} in the appendix.
The following lemma shows that identifying the best values for $\lambda$ and $x$ reduces to maximizing a quadratic function.

\begin{lemma}[Symmetric nonnegative rank-one approximation]
\label{lem:optimal-symetric-rank1approx}
Let $C\in\mathbb{R}^{N\times N}$ be a symmetric positive semidefinite matrix.
Then
\[
\min_{\substack{\lambda\geq0,\;x\geq0\\\|x\|_2=1}}\|C-\lambda xx^\top\|_F^2 = \|C\|_F^2-\left(\max_{\substack{x\geq0\\\|x\|_2=1}}x^\top Cx\right)^2.
\]
\end{lemma}

\begin{proof}
Since $\|xx^\top\|_F^2=\|x\|_2^4=1$, completing the square gives
\[
\|C-\lambda xx^\top\|_F^2 = \|C\|_F^2+\lambda^2-2\lambda x^\top Cx = \|C\|_F^2-(x^\top Cx)^2 +(\lambda-x^\top Cx)^2.
\]
As $C\succeq0$, the choice $\lambda=x^\top Cx\geq0$ is feasible and minimizes the last term.
Hence, finding a best symmetric nonnegative rank-one approximation of $C$ is equivalent to maximizing $x^\top Cx$ over the set of nonnegative unit vectors.
\end{proof}
Define
\begin{equation}\label{eq:mu_definition}
\mu(C):=\max_{\substack{x\geq0\\\|x\|_2=1}}x^\top Cx.
\end{equation}
By Lemmas~\ref{lem:optimal-symetric-rank1approx} and \ref{lem:rank1-symmetrization}, the optimal residual error is
\begin{equation}\label{eq:optimal-residual-error}
\min_{\substack{Z\geq0\\\operatorname{rank}(Z)\leq1}} \|C-Z\|_F^2 = \|C\|_F^2-\mu(C)^2.
\end{equation}
To identify $\mu(C)$, we use the continuous characterization of the clique number established by Belachew and Gillis~\cite{belachew2017solving}. Specifically, for every $b\geq N$, they showed that
\[
\omega(G)=\max_{z\geq0,\|z\|_2=1} z^TM_bz, \quad \text{with }\quad M_b:=(1+b)(A+I_N)-b\mathbf{1}_{N\times N}.
\]
In our case, we have
\[
C=dI_N+A-\frac{1}{\alpha}\mathbf{1}_{N\times N}=(d-1)I_N+\frac{1}{N+1}\left((1+N)(A+I_N)-N\mathbf{1}_{N\times N}\right),
\]
hence
\[
\mu(C)=\max_{x\geq0,\|x\|_2=1} x^TCx=d-1+\frac{\omega(G)}{N+1},
\]
and
\begin{equation}\label{eq:bestapproxC}
\min_{\lambda\geq0,x\geq0,\|x\|_2=1} \|C-\lambda xx^\top\|_F^2=\|C\|_F^2-\left(d-1+\frac{\omega(G)}{N+1}\right)^2.
\end{equation}
This identity motivates the threshold $T$ defined in~\eqref{eq:reduction}: the optimal residual error is at most $T$ if and only if $\omega(G)\geq J$.
We can then see the core idea of the different elements of the reduction~\eqref{eq:reduction}: the clique number of $G$ is encoded in the optimal nonnegative rank-one approximation error of $C$.

However, up to now, this correspondence has only been established for rank-$2$ approximations of $X$ that keep $vv^\top$ fixed.
An arbitrary nonnegative rank-$2$ approximation may achieve a smaller error by varying both factors.
In the next subsections, we show that this improvement cannot overcome the error gap between consecutive clique sizes, namely $J-1$ and $J$.
Section~\ref{subsec:numerical-illustration} provides a numerical illustration of this phenomenon and of the structure induced by the reduction.

\subsection{Bounds for general rank-two approximations}
\label{subsec:r2nmf-general-bounds}
In this subsection, we compare the error of a general nonnegative rank-two approximation of $X$ with the optimal error obtained when $vv^\top$ is kept fixed and the remaining rank-one term is chosen optimally, as given in~\eqref{eq:optimal-residual-error}.
The argument relies on two results.
The first result bounds the possible improvement obtained by allowing a bounded negative part in a rank-$1$ approximation.
The second controls how errors in the last row and column affect the residual in the leading block.
These two results are then combined in Lemma~\ref{lem:core-lemma}.

In Lemma~\ref{lem:optimal-symetric-rank1approx}, we considered the best nonnegative rank-$1$ approximation of a symmetric positive semidefinite matrix $C$.
In the next lemma, we now relax this nonnegativity constraint by allowing the rank-$1$ approximation $xy^\top$ to have negative entries, subject to the entrywise lower bound $xy^\top\geq-K$, where $K\in\mathbb{R}_{\geq0}^{N\times N}$ is a matrix specifying the magnitude of the negativity allowed at each entry.
We show that this relaxation can reduce the optimal rank-$1$ error in~\eqref{eq:optimal-residual-error} by at most $2\langle C^-,K\rangle$.
This last quantity depends only on the negativity allowed at the entries where $C$ is negative, as already suggested in Section~\ref{subsec:idea}.

\begin{lemma}[Robust rank-$1$ approximation bound]
\label{lem:rank1negative}
Let $C\in\mathbb{R}^{N\times N}$ be a symmetric positive semidefinite matrix.
Let $K\in\mathbb{R}_{\geq0}^{N\times N}$.
Then every pair of vectors $x,y\in\mathbb{R}^N$ satisfying $xy^\top\geq-K$ entrywise is such that
\begin{equation}\label{eq:rank1approx_with_negativ}
\|C-xy^\top\|_F^2\geq\|C\|_F^2-\mu(C)^2-2\langle C^-,K\rangle,
\end{equation}
where $C^-=\max(-C,0)$ denotes the entrywise negative part of $C$.
\end{lemma}

\begin{proof}
Since $C\succeq0$, using~\eqref{eq:mu_definition} and the Cauchy-Schwarz inequality gives, for any nonnegative vectors $u,v$,
\[
u^\top Cv \leq \sqrt{u^\top Cu}\sqrt{v^\top Cv}\leq \mu(C)\|u\|_2\|v\|_2.
\]
The positive part of $xy^\top$ satisfies $(xy^\top)^+=x^+(y^+)^\top+x^-(y^-)^\top$.
Consequently,
\[
\langle C,(xy^\top)^+\rangle\leq\mu(C)\left(\|x^+\|_2\|y^+\|_2+\|x^-\|_2\|y^-\|_2\right)\leq\mu(C)\|x\|_2\|y\|_2,
\]
where the last inequality follows again from Cauchy--Schwarz and the identities $\|x^+\|_2^2+\|x^-\|_2^2=\|x\|_2^2$ and $\|y^+\|_2^2+\|y^-\|_2^2=\|y\|_2^2$.

For the negative part, the assumption $xy^\top\geq-K$ gives $0\leq(xy^\top)^-\leq K$.
Hence,
\[
-\langle C,(xy^\top)^-\rangle \leq \langle C^-,(xy^\top)^-\rangle \leq \langle C^-,K\rangle.
\]
Combining these bounds gives
\[
\langle C,xy^\top\rangle \leq \mu(C)\|x\|_2\|y\|_2+\langle C^-,K\rangle.
\]
Using $\|xy^\top\|_F^2=\|x\|_2^2\|y\|_2^2$ and completing the square, we obtain
\[
\begin{aligned}
\|C-xy^\top\|_F^2 & = \|C\|_F^2-2\langle C,xy^\top\rangle+\|x\|_2^2\|y\|_2^2 \\
& \geq \|C\|_F^2 -2\mu(C)\|x\|_2\|y\|_2 +\|x\|_2^2\|y\|_2^2 -2\langle C^-,K\rangle \\
& = \|C\|_F^2-\mu(C)^2 -2\langle C^-,K\rangle + \bigl(\|x\|_2\|y\|_2-\mu(C)\bigr)^2.
\end{aligned}
\]
Dropping the final nonnegative term gives \eqref{eq:rank1approx_with_negativ}.
\end{proof}

The previous lemma bounds the possible improvement obtained by allowing a bounded negative part in a rank-$1$ correction of $C$.
This bound will be used to control the error of a general nonnegative rank-$2$ approximation of $X$.
We write such a rank-$2$ NMF in block form as
\begin{equation}\label{eq:generalrank2approx}
Y = WH = \begin{pmatrix} D & p \\ q^\top & \delta \end{pmatrix}, \quad W\geq0, H\geq 0, 
\end{equation}
and denote the part of the error on the last row and column by
\begin{equation}\label{eq:def-B}
\beta=\|d\mathbf{1}-p\|_2^2+\|d\mathbf{1}-q\|_2^2+(\alpha d^2-\delta)^2.
\end{equation}
The total approximation error is then
\[
\|X-WH\|_F^2 = \left\|C+\frac1\alpha\mathbf{1}_{N\times N}-D\right\|_F^2 + \beta.
\]
For $Y$ as in~\eqref{eq:generalrank2approx}, if $\delta\neq0$, the unique rank-one matrix reproducing the last row and column of $Y$ exactly is
\[
\begin{pmatrix} pq^\top/\delta & p \\ q^\top &  \delta\end{pmatrix}.
\]
Thus, $pq^\top/\delta$ is the leading block induced by the last row and column of $Y$.
Since the corresponding block induced by the last row and column of $X$ is $\frac{1}{\alpha}\mathbf{1}_{N\times N}$, when $\delta\neq0$ we can define their difference as
\begin{equation}\label{eq:R_definition}
R:=\frac1\alpha\mathbf{1}_{N\times N}-\frac{pq^\top}{\delta}.
\end{equation}
The next lemma shows that a small value of $\beta$ guarantees $\delta>0$ and provides a bound on $\|R\|_F$.
\begin{lemma}[Perturbation of the last row and column]\label{lem:residual-lemma}
Let $p,q\in\mathbb{R}^N$ and $\delta\in\mathbb{R}$, and let $\beta=\|d\mathbf{1}-p\|_2^2+\|d\mathbf{1}-q\|_2^2+(\alpha d^2-\delta)^2$.
With $d$ and $\alpha$ as in~\eqref{eq:reduction}, if $\beta\leq2Nd^2$, then 
\[
\delta\geq d^2 >0 \quad\text{ and }\quad \left\|R\right\|_F\leq\frac{4\sqrt{N\beta}}{d}.
\]

\end{lemma}
\begin{proof}
First, since $(\alpha d^2-\delta)^2\leq \beta\leq2Nd^2$, we have $\alpha d^2-\delta\leq\sqrt{2N}d\leq \frac{d^2}{N}$, hence $\delta\geq\alpha d^2-\frac{d^2}{N}=d^2$.
Second, expanding $pq^\top=\bigl(d\mathbf{1}-(d\mathbf{1}-p)\bigr)\bigl(d\mathbf{1}-(d\mathbf{1}-q)\bigr)^\top$ in~\eqref{eq:R_definition} gives
\[
\delta R=-\frac{\alpha d^2-\delta}{\alpha}\mathbf{1}_{N\times N}+(d\mathbf{1}-p)(d\mathbf{1})^\top+(d\mathbf{1})(d\mathbf{1}-q)^\top-(d\mathbf{1}-p)(d\mathbf{1}-q)^\top .
\]
Each of the quantities $|\alpha d^2-\delta|$, $\|d\mathbf{1}-p\|_2$ and $\|d\mathbf{1}-q\|_2$ is at most $\sqrt{\beta}$, and $\|d\mathbf{1}\|=d\sqrt N$. Hence, using $1/\alpha <1$ and $\beta\leq\sqrt{2N}d\sqrt{\beta}$, we obtain
\begin{equation}\label{eq:deltaR}
\|\delta R\|_F\leq N\sqrt{\beta}+2d\sqrt N\sqrt{\beta}+\beta\leq\bigl(N+(2+\sqrt2)\,d\sqrt N\bigr)\sqrt{\beta}\leq4\sqrt N\,d\sqrt{\beta} .
\end{equation}
Dividing by $\delta\geq d^2$ gives the claim.
\end{proof}

We can now combine the two preceding results.
The idea is to decompose the leading block $D$ of~\eqref{eq:generalrank2approx} into the rank-$1$ block determined by the last row and column of $Y$ and a remaining correction, and then verify that this correction satisfies the assumptions of Lemma~\ref{lem:rank1negative}.
More precisely, when $\beta\leq2Nd^2$, Lemma~\ref{lem:residual-lemma} gives $\delta>0$, and we can therefore define
\begin{equation}\label{eq:def-KSR}
K=\frac{pq^\top}{\delta}, \qquad S=D-K.
\end{equation}
With these definitions,
\[
C+\frac1\alpha\mathbf{1}_{N\times N}-D=C-S+R.
\]
This last relation means that the error in the leading block can be expressed in terms of the rank-$1$ correction $S$ and the perturbation $R$.
Since $Y\geq0$, we have $K\geq0$ and $S\geq-K$.
Moreover, $S$ has rank at most one since we know that
\[
\operatorname{rank}\begin{pmatrix} D & p \\ q^T & \delta \end{pmatrix} = 1+\operatorname{rank}\left(D-\frac{pq^\top}{\delta}\right).
\]
Therefore, $S=xy^\top$ for some $x,y\in\mathbb{R}^N$, with $xy^\top\geq-K$, so Lemma~\ref{lem:rank1negative} can provide a lower bound on $\|C-S\|_F^2$.
Lemma~\ref{lem:residual-lemma} controls $\|R\|_F$ in terms of $\beta$.

The following lemma combines these estimates to show that a general rank-$2$ NMF of $X$ can improve the optimal nonnegative rank-$1$ error for $C$ (see~\eqref{eq:bestapproxC}) by at most $20N^2$, independently of $d$.

\begin{lemma}\label{lem:core-lemma}
Let $X$ and $C$ be the matrices constructed above.
For any nonnegative factors $W\in\mathbb{R}_{\geq0}^{(N+1)\times2}$ and $H\in\mathbb{R}_{\geq0}^{2\times(N+1)}$, we have
\[
\|X-WH\|_F^2 \geq \|C\|_F^2-\mu(C)^2-20N^2.
\]
\end{lemma}

\begin{proof}
The claim is immediate if $\|X-WH\|_F^2\geq\|C\|_F^2-\mu(C)^2$.
We therefore assume that $\|X-WH\|_F^2<\|C\|_F^2-\mu(C)^2\leq\|C\|_F^2$.

Let $Y$ and $\beta$ be defined as in~\eqref{eq:generalrank2approx} and~\eqref{eq:def-B}.
The entries of $C$ give
\[
\|C\|_F^2 \leq N\left(d-\frac1\alpha\right)^2 + \frac{N(N-1)}{\alpha^2} \leq Nd^2.
\]
Hence $\beta\leq\|X-WH\|_F^2\leq\|C\|_F^2\leq Nd^2<2Nd^2$, so Lemma~\ref{lem:residual-lemma} gives $\delta\geq d^2>0$.
The matrices $K$ and $S$ in~\eqref{eq:def-KSR} and the
perturbation $R$ in~\eqref{eq:R_definition} are therefore well defined.
The approximation error can be written as
\begin{equation}\label{eq:termstobound}
\|X-WH\|_F^2 =\|C-S+R\|_F^2+\beta =\|C-S\|_F^2+2\langle C-S,R\rangle+\|R\|_F^2+\beta.
\end{equation}
We bound the first two terms separately.
\begin{itemize}
    \item By the triangle inequality,
    \[
    \|C-S\|_F\leq\|C-S+R\|_F+\|R\|_F\leq\|C\|_F+\|R\|_F,
    \]
    since $\|C-S+R\|_F^2\leq\|X-WH\|_F^2\leq\|C\|_F^2$.
    Hence, the cross term in \eqref{eq:termstobound} is bounded as follows:
    \begin{equation}\label{eq:term2}
        \langle C-S,R\rangle \geq - \|C-S\|_F\|R\|_F\geq-\|R\|_F\bigl(\|C\|_F+\|R\|_F\bigr).
    \end{equation}

    \item As shown above, $K\geq0$, $\operatorname{rank}(S)\leq1$, and $S\geq-K$, so Lemma~\ref{lem:rank1negative} applies to $S$:
    \[
    \|C-S\|_F^2 \geq\|C\|_F^2-\mu(C)^2-2\langle C^-,K\rangle.
    \]
    As defined in~\eqref{eq:C-positive-negative}, $C^-=\frac{1}{\alpha}\bigl(\mathbf{1}_{N\times N}-I_N-A\bigr)$, so $\|C^-\|_F^2 = \frac{1}{\alpha^2}(N^2-N-2|E|)<N^2$.
    Using $K=\frac{1}{\alpha}\mathbf{1}_{N\times N}-R$, we obtain
    \begin{equation}\label{eq:term1}
        2\langle C^-,K\rangle = \frac2\alpha\langle C^-,\mathbf{1}_{N\times N}\rangle-2\langle C^-,R\rangle\leq 2N^2+2N\|R\|_F.
    \end{equation}
\end{itemize}
Plugging~\eqref{eq:term2} and \eqref{eq:term1} into \eqref{eq:termstobound} gives
\[
\|X-WH\|_F^2 \geq\|C\|_F^2-\mu(C)^2-2N^2-2\|R\|_F(\|C\|_F+N)-\|R\|_F^2+\beta.
\]
Lemma~\ref{lem:residual-lemma} gives $\|R\|_F\leq4\sqrt{N\beta}/d$, hence
\[
\|X-WH\|_F^2 \geq\|C\|_F^2-\mu(C)^2-2N^2-\frac{8\sqrt{N\beta}}{d}(\|C\|_F+N)+\left(1-\frac{16N}{d^2}\right)\beta.
\]
Since $d=20N^3$ and $N\geq2$, we have
\[
\|C\|_F+N
\leq\left(1+\frac{\sqrt N}{d}\right)\sqrt{N}d\leq\frac{21}{20}\sqrt{N}d,\qquad1-\frac{16N}{d^2}\geq\frac{49}{50}.
\]
Consequently,
\[
\|X-WH\|_F^2 \geq\|C\|_F^2-\mu(C)^2-2N^2 -\frac{42}{5}N\sqrt{\beta}+\frac{49}{50}\beta.
\]
Completing the square gives
\[
\frac{49}{50}\beta-\frac{42}{5}N\sqrt{\beta} = \frac{49}{50} \left(\sqrt{\beta}-\frac{30}{7}N\right)^2 -18N^2 \geq-18N^2.
\]
Therefore,
\[
\|X-WH\|_F^2\geq\|C\|_F^2-\mu(C)^2-20N^2,
\]
as claimed.
\end{proof}

\subsection{Correctness of the reduction}
\label{subsec:r2nmf-correctness}
Lemma~\ref{lem:core-lemma} shows that no nonnegative rank-$2$ approximation can improve the optimal nonnegative rank-$1$ error for $C$ by more than $20N^2$.
Crucially, this bound is independent of $d$.
In order to prove the correctness of the reduction, we show that the gap between the residual errors associated with consecutive clique sizes increases with $d$.
Choosing $d$ so that this gap exceeds $20N^2$ therefore ensures that an approximation can reach the threshold $T$ only if $G$ contains a clique of size at least $J$.

\begin{proposition}[Correctness of the reduction]\label{prop:reduction}
The instance $(X,T)$ is a yes-instance of \textsc{Rank-2-NMF} if and only if $(G,J)$ is a yes-instance of \textsc{Clique}.
\end{proposition}

\begin{proof}
\textbf{If part.} Suppose first that $(G,J)$ is a yes-instance of \textsc{Clique}.
This means that $G$ contains a clique with $J$ vertices.
From such a clique, we define $z\in\mathbb{R}^{N}$ by
\[
z_i =
\begin{cases}
1/\sqrt{J}, & \text{if vertex } i \text{ belongs to the clique},\\
0, & \text{otherwise}.
\end{cases}
\]
Since $z$ has exactly $J$ nonzero entries, each equal to $1/\sqrt{J}$, we have $\|z\|_2=1$.
Moreover, its support is a clique; then,
\[
z^{\top}Cz= \frac{1}{J}\sum_{\substack{i,j\\(i,j)\in\text{clique}}} C_{ij}=\frac{1}{J}\left(J\left(d-\frac{1}{\alpha}\right) + J(J-1)\left(1-\frac{1}{\alpha}\right)\right)= d-1+\frac{J}{N+1}.
\]
We also define
\[
v=\begin{pmatrix} \frac{1}{\sqrt{\alpha}}\mathbf1_{N}\\
\sqrt{\alpha}d \end{pmatrix},
\quad \widetilde{z}= \begin{pmatrix}z\\0\end{pmatrix},
\quad W=\begin{pmatrix}v & \sqrt{d-1+\frac{J}{N+1}}\widetilde z\end{pmatrix}, \quad H=W^\top.
\]
These factors are nonnegative and verify
\[
\|X-WH\|_F^2 =  \left\|C-\left(d-1+\frac{J}{N+1}\right)zz^\top \right\|_F^2 = \|C\|_F^2-\left(d-1+\frac{J}{N+1}\right)^2=T.
\]
Thus, $(X,T)$ is a yes-instance of \textsc{Rank-2-NMF}.

\paragraph{Only if part.} Suppose now that $(X,T)$ is a yes-instance of \textsc{Rank-2-NMF}.
This means that there exist $W\in\mathbb{R}_{\geq 0}^{n\times 2}$, $H\in\mathbb{R}_{\geq 0}^{2\times n}$ such that $\|X-WH\|_F^2\leq T$.
We must show the existence of a clique of size at least $J$.
Assume, to the contrary, that $\omega(G)\leq J-1$.
By Lemma~\ref{lem:core-lemma},
\[
\begin{aligned}
\|X-WH\|_F^2 & \geq \|C\|_F^2- \left(d-1+\frac{\omega(G)}{N+1}\right)^2 -20N^2 \\
& \geq \|C\|_F^2-\left(d-1+\frac{J-1}{N+1}\right)^2-20N^2\\
& = \|C\|_F^2-\left(d-1+\frac{J}{N+1}\right)^2+\left(d-1+\frac{J}{N+1}\right)^2-\left(d-1+\frac{J-1}{N+1}\right)^2-20N^2\\
& = T+\left(d-1+\frac{J}{N+1}\right)^2-\left(d-1+\frac{J-1}{N+1}\right)^2-20N^2.
\end{aligned}
\]
Since
\[
\left(d-1+\frac{J}{N+1}\right)^2 - \left(d-1+\frac{J-1}{N+1}\right)^2= \frac{2(d-1)}{N+1} + \frac{2J-1}{(N+1)^2} > \frac{2(d-1)}{N+1}\geq \frac{d}{N}=20N^2,
\]
we obtain
\[
\|X-WH\|_F^2>T,
\]
a contradiction.
Hence $\omega(G)\geq J$, and $(G,J)$ is a yes-instance of \textsc{Clique}.
\end{proof}

Our main result, Theorem~\ref{thm:r2hardness}, follows directly from the reduction
established above.

\begin{proof}[Proof of Theorem~\ref{thm:r2hardness}]
The instance $(X,T)$ is constructed in polynomial time and \textsc{Clique} is NP-complete~\cite{karp1972reducibility}, so Proposition~\ref{prop:reduction} implies the claim.

The reduction also establishes strong NP-hardness.
Indeed, all numerical quantities in the constructed instance $(X,T)$ have magnitudes, numerators, and denominators bounded polynomially in $N$.
Hence, \textsc{Rank-2-NMF} is strongly NP-hard.
\end{proof}

\subsection{A numerical illustration}
\label{subsec:numerical-illustration}

We illustrate the reduction on a graph $G$ with $N=6$ vertices whose adjacency matrix is
\[
A=
\begin{pmatrix}
0&1&1&1&0&1\\
1&0&1&0&1&0\\
1&1&0&1&0&1\\
1&0&1&0&0&1\\
0&1&0&0&0&1\\
1&0&1&1&1&0
\end{pmatrix}.
\]
For this graph, we have $\omega(G)=4$ and a unique maximum clique,
\[
\mathcal{C}=\{1,3,4,6\}.
\]
For $N=6$, our construction~\eqref{eq:reduction} gives
\[
d=20N^3=4320, \qquad \alpha=\frac{7}{6},
\]
and hence
\[
X=
\begin{pmatrix}
4320&1&1&1&0&1&4320\\
1&4320&1&0&1&0&4320\\
1&1&4320&1&0&1&4320\\
1&0&1&4320&0&1&4320\\
0&1&0&0&4320&1&4320\\
1&0&1&1&1&4320&4320\\
4320&4320&4320&4320&4320&4320&21772800
\end{pmatrix}.
\]
In order to illustrate the gap trick used in the proof of Proposition~\ref{prop:reduction}, we define
\[
T_j:=\|C\|_F^2-\left(d-1+\frac{j}{N+1}\right)^2.
\]
For our example, taking $J=4$, we have $T=T_4 = \|C\|_F^2-\mu(C)^2$ and the different values are given in Table~\ref{tab:numerical-error-levels}.
\begin{table}[ht]
\centering
\begin{tabular}{c|c|c}
\hline
$j$ & $T_j$ & $T_{j-1}-T_j$\\
\hline
1 & $93\,274\,982.857$ & --\\
2 & $93\,273\,748.796$ & $1\,234.061$\\
3 & $93\,272\,514.694$ & $1\,234.102$\\
4 & $93\,271\,280.551=T$ & $1\,234.143$\\
\hline
\end{tabular}
\caption{Residual error levels associated with clique sizes $j=1,\ldots,4$.}
\label{tab:numerical-error-levels}
\end{table}

In Table~\ref{tab:numerical-error-levels}, we see that the different levels are separated by approximately $1234$ while the maximum improvement allowed by Lemma~\ref{lem:core-lemma} is $20N^2=720$.
For an optimal approximation $Y^{\ast}=W^{\ast}H^{\ast}$, together with Lemma~\ref{lem:core-lemma}, we then have
\[
T_4-20N^2=93\,270\,560.551 \leq \|X-Y^{\ast}\|_F^2 \leq T_4=T=93\,271\,280.551.
\]
Using a heuristic algorithm~\cite{vandaele2016heuristics}, we were able to obtain nonnegative
factors $W$ and $H$ such that the product $Y=WH$ is
\begin{equation}\label{eq:solexample}
Y\approx
\begin{pmatrix}
1080.748&0.855&1080.748&1080.752&0.855&1080.748&4320.000\\
0.855&0.857&0.855&0.855&0.857&0.855&4320.857\\
1080.748&0.855&1080.749&1080.752&0.855&1080.748&4320.000\\
1080.752&0.855&1080.752&1080.756&0.855&1080.752&4320.000\\
0.855&0.857&0.855&0.855&0.857&0.855&4320.857\\
1080.748&0.855&1080.748&1080.752&0.855&1080.748&4320.000\\
4320.000&4320.857&4320.000&4320.000&4320.857&4320.000&21772800.000
\end{pmatrix}.
\end{equation}
This numerical factorization has an objective value of
\[
\|X-Y\|_F^2\approx 93\,271\,277.58<T.
\]
In fact, the numerical error is approximately $2.97$ below $T$.
Thus, as discussed in Lemma~\ref{lem:core-lemma}, optimizing both rank-one terms can slightly decrease the objective compared with the approximation obtained by keeping $vv^\top$ fixed.

We now examine the decomposition used in the proof of
Lemma~\ref{lem:core-lemma}. Writing
\[
Y=
\begin{pmatrix}D&p \\ q^\top&\delta\end{pmatrix},
\]
we form
\[
K=\frac{pq^\top}{\delta},\qquad S=D-K,\qquad R=\frac1\alpha\mathbf{1}_{N\times N}-K.
\]
For the solution~\eqref{eq:solexample}, we obtain
\[
\beta\approx 2.9364,\qquad\text{ and }\qquad\|R\|_F\approx9.62\times10^{-4}.
\]
Hence the last row and column of $Y$ induce a leading rank-one block that is very close to $\frac1\alpha\mathbf{1}_{N\times N}$, as predicted by Lemma~\ref{lem:residual-lemma}.
The remaining rank-one correction is
\[
S=D-K\approx
\begin{pmatrix}
1079.891&-0.002&1079.891&1079.895&-0.002&1079.891\\
-0.002&0&-0.002&-0.002&0&-0.002\\
1079.891&-0.002&1079.891&1079.895&-0.002&1079.891\\
1079.895&-0.002&1079.895&1079.899&-0.002&1079.895\\
-0.002&0&-0.002&-0.002&0&-0.002\\
1079.891&-0.002&1079.891&1079.895&-0.002&1079.891
\end{pmatrix}.
\]
The clique structure is now apparent: the entries of $S$ indexed by $\mathcal C\times\mathcal C$, with $\mathcal{C}=\{1,3,4,6\}$, are all approximately $1079.89$, while all entries involving vertices $2$ or $5$ have magnitude below $3\times10^{-3}$.

This is precisely the structure suggested by the rank-one correction associated with the maximum clique: for its normalized indicator,
\[
z=\frac{1}{2}(1,0,1,1,0,1)^\top,
\]
the correction obtained by keeping $vv^\top$ fixed is
\[
\mu(C)zz^\top,\qquad\mu(C)=d-1+\frac{4}{7}\approx4319.5714,
\]
whose nonzero entries are all equal to
\[
\frac{\mu(C)}{4} \approx 1079.8929.
\]
The matrix
$S=D-pq^\top/\delta$ obtained from this numerical factorization is extremely close to the rank-one correction associated with the maximum clique and  reveals the clique $\{1,3,4,6\}$.

\section{Extension to symmetric NMF}
\label{sec:symnmf}
Symmetric NMF (SymNMF) seeks an approximation $X\approx WW^\top$ of a symmetric nonnegative matrix $X$ with $W\geq0$; it is used in particular for graph clustering~\cite{kuang2012symmetric}.
SymNMF is NP-hard when $r$ is part of the input, since checking whether $X$ belongs to the cone of completely positive matrices, that is, whether $X=WW^\top$ for some nonnegative matrix $W$, is NP-hard~\cite{dickinson2014computational}.
Its fixed-rank complexity follows the same pattern as that of NMF (see Section~\ref{subsec:nutshell}).
For $r=1$, the problem is easy: by the Perron--Frobenius theorem, a symmetric nonnegative matrix admits a nonnegative leading eigenvector, so that the best rank-$1$ approximation of the form $ww^\top$, given by the eigendecomposition, is already nonnegative.
For input matrices with negative entries, however, the symmetric rank-$1$ problem is already NP-hard, as shown by Belachew and Gillis through its connection with the maximum clique problem~\cite{belachew2017solving}.
For SymNMF itself, where the input matrix is nonnegative, the first nontrivial case is therefore $r=2$.
In this section, we observe that the reduction of Section~\ref{sec:rank-two-hardness} settles this case without any modification.
We first establish an additional property of the
matrices produced by this reduction.

\begin{proposition}\label{prop:complete-positivity}
The matrix $X$ defined in~\eqref{eq:reduction} is completely positive.
\end{proposition}
\begin{proof}
First consider the leading block $dI_N+A$ of $X$.
Let $D_G=\operatorname{diag}(\deg(1),\ldots,\deg(N))$ denote the degree matrix of $G$.
Recall that the signless Laplacian matrix $D_G+A$ is completely positive, as can be seen from the decomposition:
\[
D_G+A = \sum_{(i,j)\in E}(e_i+e_j)(e_i+e_j)^\top,
\]
where $e_1,\ldots,e_N$ denote the standard basis vectors of $\mathbb{R}^N$.
We can therefore write
\[
dI_N+A=D_G+A+(dI_N-D_G)=\sum_{(i,j)\in E}(e_i+e_j)(e_i+e_j)^\top+\sum_{i=1}^N\bigl(d-\deg(i)\bigr)e_ie_i^\top.
\]
Thus, the leading block of $X$ is completely positive since $d\geq\deg(i)$ for every $i$.
We next extend these rank-one terms to $\mathbb{R}^{N+1}$, preserving the leading block while matching the off-diagonal entries in the last row and column of $X$.
In the following, $e_i$ denotes the $i$th standard basis vector of $\mathbb{R}^{N+1}$.
We define
\[
X^{(1)}=\sum_{(i,j)\in E}(e_i+e_j+e_{N+1})(e_i+e_j+e_{N+1})^\top, \quad X^{(2)}=\sum_{i=1}^N\bigl(d-\deg(i)\bigr)(e_i+e_{N+1})(e_i+e_{N+1})^\top,
\]
and we obtain
\[
X^{(1)}+X^{(2)}=\begin{pmatrix}dI_N+A & d\mathbf{1}_N\\d\mathbf{1}_N^\top & Nd-|E|\end{pmatrix}.
\]
Thus, $X^{(1)}+X^{(2)}$ agrees with $X$ except for its lower-right entry.
Since $\alpha>1$ and $d=20N^3\geq N$, we have $\alpha d^2-Nd+|E|\geq d(d-N)\geq0$.
We can define
\[
X^{(3)}=\bigl(\alpha d^2-Nd+|E|\bigr)e_{N+1}e_{N+1}^\top,
\]
which gives
\[
X=X^{(1)}+X^{(2)}+X^{(3)}.
\]
Each of the matrices $X^{(1)}$, $X^{(2)}$, and $X^{(3)}$ is completely positive by construction.
Hence, $X$ is completely positive.
\end{proof}
Thus, even though completely positive matrices admit exact symmetric nonnegative factorizations, finding a best rank-two symmetric nonnegative approximation remains strongly NP-hard, as shown by the following corollary.
\begin{corollary}
\label{cor:rank2-symnmf}
Given a symmetric matrix $X\in\mathbb{Q}_{\geq0}^{n\times n}$ and a rational threshold $T\geq0$, deciding whether there exists $W\in\mathbb{R}_{\geq0}^{n\times2}$ such that
\[
\|X-WW^\top\|_F^2\leq T
\]
is strongly NP-hard, even when $X$ is completely positive.
\end{corollary}

\begin{proof}
Let $(X,T)$ be the instance constructed in Section~\ref{subsec:r2nmf-construction} from an instance $(G,J)$ of \textsc{Clique}, with $n=N+1$.
The matrix $X$ is symmetric, nonnegative, and, by Proposition~\ref{prop:complete-positivity}, completely positive.
Since this construction is carried out in polynomial time, it suffices to show that there exists $W\in\mathbb{R}_{\geq0}^{(N+1)\times2}$ with $\|X-WW^\top\|_F^2\leq T$ if and only if $(G,J)$ is a yes-instance of \textsc{Clique}.

If $(G,J)$ is a yes-instance, the factors constructed in the proof of Proposition~\ref{prop:reduction} satisfy $H=W^\top$ and $\|X-WW^\top\|_F^2=T$.
Conversely, any $W\geq0$ satisfying $\|X-WW^\top\|_F^2\leq T$ gives a feasible rank-$2$ NMF by setting $H=W^\top$, hence Proposition~\ref{prop:reduction} shows that $\omega(G)\geq J$, that is, $(G,J)$ is a yes-instance of \textsc{Clique}.
\end{proof}

\section{Conclusion}

We proved that computing a best rank-$2$ NMF in the Frobenius norm is strongly NP-hard by providing a polynomial-time reduction from \textsc{Clique}.
The hardness already holds for symmetric positive-definite matrices with nonnegative integer entries.
The same reduction also establishes the strong NP-hardness of rank-$2$ symmetric NMF.

We identify two interesting directions for further investigation.
First, if the unconstrained best rank-$2$ approximation $X_2$ is nonnegative, then it is also optimal for the nonnegative problem.
It would therefore be interesting to determine the complexity when $X_2$ has only a fixed number of negative entries, in particular when it has a single negative entry.
Second, our reduction uses input matrices with entries as large as $\alpha d^2=400N^5(N+1)=\Theta(N^6)$.
It is therefore natural to ask whether rank-$2$ NMF remains NP-hard for nonnegative integer input matrices whose entries
are bounded by an absolute constant, and in particular for
binary input matrices.

\section*{Acknowledgment} 
The author thanks Nicolas Gillis for helpful discussions and for providing feedback on several versions of this manuscript.
In this work, AI was used to assist with the development and refinement of proof arguments, especially inequalities~\eqref{eq:deltaR} and~\eqref{eq:term2}.
All mathematical arguments, reported computations, and references retained in the manuscript have been checked by the author. 
The author assumes responsibility for all content.

\newpage

\bibliographystyle{plain}
\bibliography{bib}

\newpage

\appendix

\section{Auxiliary results}
\label{app:auxiliary-results}

\subsection{Symmetrization of nonnegative rank-one approximations}
\label{app:rank1-symmetrization}

The following lemma shows that we can restrict the nonnegative rank-one approximation of a positive semidefinite matrix to be symmetric.

\begin{lemma}[Symmetrization of nonnegative rank-one approximations]
\label{lem:rank1-symmetrization}
Let $C\in\mathbb{R}^{N\times N}$ be symmetric positive semidefinite.
The problem
\[
\min_{\substack{Z\in\mathbb{R}_{\geq0}^{N\times N}\\\operatorname{rank}(Z)\leq1}} \|C-Z\|_F^2
\]
admits an optimal solution of the form $\lambda xx^\top$, where $\lambda\geq0$, $x\geq0$, and $\|x\|_2=1$.
\end{lemma}

\begin{proof}
Let $Z$ be an optimal solution, and write $Z=\lambda xy^\top$, where $\lambda\geq0$, $x,y\geq0$, and $\|x\|_2=\|y\|_2=1$.
Since $C$ is symmetric, by expanding the squared norms, we can show that
\[
\|C-\lambda xy^\top\|_F^2 =  \frac{1}{2}\|C-\lambda xx^\top\|_F^2 +\frac{1}{2}\|C-\lambda yy^\top\|_F^2+\lambda(x-y)^\top C(x-y).
\]
The final term is nonnegative because $C\succeq0$ and $\lambda\geq0$.
Since the minimum of two quantities does not exceed the average, we have that
\[
\min\left\{\|C-\lambda xx^\top\|_F^2,\|C-\lambda yy^\top\|_F^2\right\}\leq\|C-\lambda xy^\top\|_F^2.
\]
Thus, replacing $\lambda xy^\top$ by the better of the two rank-$1$ matrices $\lambda xx^\top$ and $\lambda yy^\top$ preserves optimality.
\end{proof}

\end{document}